\documentclass[11pt, reqno]{amsart}
\usepackage{amsmath,amsfonts,amssymb,amsthm,color}
\usepackage{epsfig}

\usepackage{bbm}
\usepackage{enumerate}  
\usepackage{amstext}
\usepackage{mathrsfs}
\usepackage{xspace}
\usepackage{amsfonts}
\usepackage{amsmath}
\usepackage{amssymb}
\usepackage{amstext}
\usepackage{amsthm}       
\usepackage{xspace}
\usepackage[active]{srcltx}
\usepackage{cite}
\usepackage{color}

\newtheorem{teorema}{Theorem}
\newtheorem{theorem}[teorema]{Theorem}
\newtheorem{prop}[teorema]{Proposition}
\newtheorem{lemma}[teorema]{Lemma}
\newtheorem{definition}[teorema]{Definition}

\newtheorem{guess}[teorema]{Remark}
\newtheorem{example}[teorema]{Example}

\newenvironment{oss}{\begin{guess} \begin{rm}}{\end{rm} \end{guess}}

\newtheorem{proposition}[teorema]{Proposition}
\newtheorem{corollary}[teorema]{Corollary}

\newcommand{\F}{\mathcal F}
\newcommand{\PP}{\mathbb P}
 \newcommand{\N}{\mathbb N}
 \newcommand{\T}{\mathbb T}
 \newcommand{\Z}{\mathbb Z}
 \newcommand{\R}{\mathbb R}

 \newcommand{\cF}{\mathcal F}

\newcommand{\cH}{\mathcal H}
  \newcommand{\cL}{\mathcal L}
\newcommand{\hcL}{\hat{\mathcal L}}

\newcommand{\parts}{\mathscr{P}}

 \newcommand{\af}{\alpha}

\newcommand{\be}{\beta}
 \newcommand{\ga}{\gamma}

  \newcommand{\lam}{\lambda}
  \newcommand{\Lam}{\Lambda}

\newcommand{\tr}{\operatorname{Tr}}

\def\AND{\text{ and }}

\def\FORALL{\text{ for all }}

\def\du#1{\langle#1\rangle}

\newcommand{\sdy}{Y_{x}}
\newcommand{\Exp}{\mathbb E}

\def\N{\mathbb{N}}
\def\Z{\mathbb{Z}}

\def\R{\mathbb{R}}

\def\T{\mathbb{T}}

\def\qed{\hfill$\square$}

\def\S{\mathcal{S}}

\def\Leg{\mathcal{L}}

\newcommand{\eps}{\varepsilon}
\newcommand{\Bor}{\mathscr{B}}

\newcommand{\D}[1]{\mbox{\rm #1}}
\newcommand{\dd}{\D{d}}

\newcommand{\cont}{\D{C}}
\newcommand{\e}{\operatorname{e}}

\newcommand{\Lip}{\D{Lip}}

\title{Discrete approximation of the viscous HJ equation}
\author[A. Davini]{Andrea Davini}
\address{Dipartimento di Matematica \\  {Sapienza} Universit\`a di Roma\\
P.le Aldo Moro 2, 00185 Roma, Italy}
\email{davini@mat.uniroma1.it}

\author[H. Ishii]{Hitoshi Ishii}
\address{Institute for Mathematics and Computer Science\\ Tsuda  University \\ 
 2-1-1 Tsuda, Kodaira, Tokyo 187-8577 Japan}
\email{hitoshi.ishii@waseda.jp}

\author[R. Iturriaga]{Renato Iturriaga}
 \address{ CIMAT\\
  A.P. 402, 3600\\
  Guanajuato, Gto. M\'{e}xico.}
\email{renato@cimat.mx}

\author[H. Sanchez]{Hector Sanchez Morgado}
 \address{Instituto de Matematicas\\
    Universidad Nacional Autonoma de Mexico, Mexico.}
 \email{hector@matem.unam.mx}

\thanks{{The work of AD, HI, and RI was partially supported by the NSF Grant No. 1440140 and the work of  HI was partially supported by the JSPS grants: KAKENHI
    \#16H03948, \#18H00833.}  The work of  HSM was partially supported by a PASPA sabatical grant from DGAPA-UNAM.  {HI thanks Hiroyoshi Mitake for his sharing 
the idea of the proof of Theorem \ref{teo eq viscous HJ}-(i) below}}

\begin{document}
\maketitle
%
%
%
%
%
\begin{abstract}
  We consider a stochastic discretization of the stationary viscous
  Hamilton-Jacobi equation on the flat $d$--dimensional torus $\T^d$ associated with a  Hamiltonian, 
  convex and superlinear in the momentum variable. 
  We show that each discrete problem admits a unique continuous solution on $\T^d$, up to additive constants.  
  By additionally assuming a technical condition on the associated Lagrangian, 
  we show that each solution of the viscous Hamilton--Jacobi equation 
  is the limit of solutions of the discrete problems, as the 
  discretization step goes to zero. 
\end{abstract}
\section*{Introduction}

Several authors have considered the approximation of the value
function in continuous time Optimal Control by means of the
value function given by a discrete time Dynamical
Programming Principle. The convergence of this approximation
is in fact the basis for computational methods of solution of the
corresponding Hamilton-Jacobi-Bellman equation. We can mention
the book \cite{bardi} and the articles \cite{CCDG,CD,CDI,CDF,G,ST} 
where convergence is proved on different settings.

In this paper, we propose a stochastic version of this discretization so to approximate the solutions of a 
viscous Hamilton--Jacobi equation of the kind
\begin{equation}\label{intro eq HJ}
-\Delta u +H(x,Du)=\alpha_0\qquad\hbox{in $\T^d$,}
\end{equation}
where $\T^d$ is the flat $d$--dimensional torus and the Hamiltonian $H:\T^d\times\R^d\to\R$ is a continuous function, convex and superlinear 
in the momentum variable. Under suitable assumptions on $H$, there is a unique real constant $\alpha_0$ such that equation 
\eqref{intro eq HJ} admits solutions in the viscosity sense. 
This constant $\alpha_0$ is often termed {\em ergodic constant} or {\em Ma\~ne critical value}. 
Furthermore, solutions to \eqref{intro eq HJ} are unique, up to additive constants, and are of class $C^2$, hence they solve 
the equation \eqref{intro eq HJ} in the classical sense. 

Solutions to \eqref{intro eq HJ} can be also regarded as fixed points, for every $t>0$, of the operator $\S(t):\cont(\T^d)\to\cont(\T^d)$, defined on 
the space $\cont(\T^d)$ of continuous $\Z^d$--periodic function on $\R^d$, as follows: 
\begin{equation}\label{eq viscous representation}
\left(\S(t)u\right)(x)=\inf_{v}\Exp\left[ u(\sdy(t))+\int_0^t \big(\,L(\sdy(s),-v(s))+\alpha_0\big)\,\dd s    \right]
\end{equation}
for every $x\in \T^{N}$ and $t>0$. Here $L$ is the Lagrangian associated to $H$ via the Legendre-Fenchel transform, 
$v:[0,\infty)\times\Omega\to\R^{N}$ is a control process satisfying suitable measurability conditions and 
$\sdy$ is the solution of the following Stochastic Differential Equation
\begin{align}\label{SDE}
\begin{cases}
d\sdy (t)=v(t)\, dt +\sqrt{2}\, d W_{t}\smallskip\\
\sdy(0)=x,
\end{cases}
\end{align}
where $W_{t}$ denotes a standard Brownian motion on $\R^d$, defined on a probability space $(\Omega,\F,\PP)$.  
In the formula \eqref{eq viscous representation}, the symbol $\Exp$ stands for the expectation with respect to 
the probability measure $\PP$ and the minimization is performed 
by letting $v$ vary in a proper class of admissible control processes.  

Motivated by this control theoretic interpretation of the viscous Hamilton--Jacobi equation, we consider the following discretization 
of the above formula \eqref{eq viscous representation}: for every fixed $\tau>0$, we introduce an operator 
$\Leg_{\tau}:\D{C}(\T^d)\to\D{C}(\T^d)$ defined as follows:
\begin{equation*}\label{intro def Leb tau}
\Leg_\tau  u(x) := \min_{q\in\R^d}\big( \tau L (x, -q) +(\eta^\tau*u)(x+\tau q)\big)\qquad\hbox{for every $x\in \R^d$,}  
\end{equation*}
where $\eta^{\tau}*u$ is the convolution of the function $u$ with the
heat kernel
\[\eta^\tau(y):=(4\pi\tau)^{-\frac d2}\e^{-\frac{|y|^2}{4\tau}}.\] 

As a preliminary fact we prove

\begin{theorem}\label{teo existence alpha tau}
Let $L:\T^d\times\R^d\to\R$ satisfy conditions {\bf (L1)-(L2)} below. Then  
there exists a unique constant $\alpha_\tau\in\R$ for which the equation 
\begin{equation}\label{eq discrete HJ}
\Leg_\tau u= u-\tau\alpha_\tau\qquad \hbox{in $\T^d$}
\end{equation}
admits a solution $u\in \D{C}(\T^d)$. {Furthermore, solutions are unique, up to additive constants.}
\end{theorem}

Our main result is the following. 

\begin{theorem}\label{teo asymptotic}
Assume  that $L:\T^d\times\R^d\to\R$ satisfy conditions {\bf (L1)-(L2)} below, together with 
\begin{equation}\label{L3}
L(x+h,q+k)+L(x-h,q-k)-2L(x,q)\leq M_\gamma (|h|^\gamma+|k|)	  \ \ \tag{\bf L3}
\end{equation}
for all $(x,q)\in\T^d\times\R^d$ and $h,k\in \overline B_{R_d}$ with $R_d:=\sqrt{d}/2$, for some constants $\gamma\in(0,\,1)$ and $M_\gamma >0$.
Let $x_0\in\T^d$ be fixed and denote by $u_\tau$ the unique solution of
\[
\Leg_\tau u= u-\tau\alpha_\tau\qquad \hbox{in $\T^d$}
\]
such that $u_\tau(x_0)=0$.
Then the family $\{ u_\tau\,\mid\,\tau\in (0,1)\}$ is equi--bounded and equi--continuous in $\cont(\T^{d})$
and $(\alpha_\tau,u_\tau)$ converges to $(\alpha_0,u)$ in $\R\times\D{C}(\T^d)$, as $\tau\to 0$, where 
$u$ is the unique viscosity solution to
\begin{equation}\label{eq2 viscous HJ}
-\Delta u +H(x,Du)=\alpha_0\qquad\hbox{in $\T^d$},
\end{equation}
satisfying $u(x_0)=0$.
\end{theorem}

The paper is organized as follows: Section \ref{sez preliminaries} contains the standing assumptions and some preliminary facts on the 
viscous Hamilton--Jacobi equation \eqref{intro eq HJ}. In Section \ref{sez discretization} we introduce the discrete operator $\Leg_\tau$ and study its 
main properties, in particular we prove Theorem  \ref{teo existence alpha tau}. In Section \ref{sez equicontinuity} we prove equi--continuity of the solutions 
of the discrete problems. Section \ref{sec:proof-main-result} is devoted to the proof Theorem \ref{teo asymptotic}, while Section \ref{sez examples} contains 
some examples for which the assertion of Theorem \ref{teo asymptotic} holds true.

\section{The viscous Hamilton--Jacobi equation}\label{sez preliminaries}
\numberwithin{equation}{section}
\numberwithin{teorema}{section}

Throughout the paper, we will call Lagrangian a continuous function $L:\R^d\to\R$, which is $\Z^d$--periodic in the space variable $x$. 
Equivalently, $L$ can be thought as defined on the tangent bundle $\T^d\times\R^d$ of the flat $d$--dimensional torus $\T^d$.  
We will assume $L$ to satisfy the following hypotheses:
\begin{itemize}\label{Tonelli}
\item[\bf (L1)] {\bf (Convexity)}\quad  for every  $x\in \R^d$, the map $q\mapsto L(x,q)$ is  
convex on $\R^d$.\smallskip
\item[\bf (L2)] {\bf (Superlinearity)}\quad $\displaystyle 
\inf_{x\in \R^d} \frac{L(x,q)}{|q|}\to +\infty\qquad\hbox{as $|q|\to +\infty$.}$
\end{itemize}
%
%
%
To any such Lagrangian, we can associate a Hamiltonian function $H:\R^d\times\R^d\to\R$ via  the {\em Legendre-Fenchel  transform}:
\begin{equation}\label{def L}
H(x,p):=\sup_{q\in\R^d}\left\{\langle p, q\rangle -
L(x,q)\right\}.
\end{equation}
Such a function $H$ is clearly $Z^d$--periodic in $x$. Furthermore, it satisfies convexity and superlinearity conditions analogous to {\bf (L1)} and {\bf (L2)}, 
to which we shall refer as  {\bf (H1)} and {\bf (H2)} in the sequel. Later in the paper, we will assume $L$ 
to satisfy the additional assumption \eqref{L3}. We shall see that this implies the following 
request on the associated Hamiltonian $H$, see Proposition \ref{H-Holder}:
\begin{itemize}
%
\item[\bf (H3)] there exists a constant {$K=K(\gamma)>0$} such that
  \[|H(x,p)-H(y,p)|\le K|x-y|^\ga\quad\hbox{ for any } x,y\in\T^d,
    p\in\R^d.
\]
\end{itemize}

We will see that under the assumptions {\bf (H1)},{\bf\,(H2)},{\bf\,(H3)}, 
there is a unique real constant $\alpha_0$ for which the equation
\begin{equation}\label{eq viscous HJ}
-\Delta u +H(x,Du)=\alpha_0\qquad\hbox{in $\T^d$},
\end{equation}
admits viscosity solutions. Such solutions are actually of class $C^2$ and unique, up to additive constants.
Our goal is to perform a discrete approximation of the solution of \eqref{eq viscous HJ}.
In the sequel, we will use the notation $\cont(\T^d)$ to denote the family of continuous functions on $\T^d$, or, equivalently, the family of continuous and $\Z^d$--periodic functions on $\R^d$, endowed with the $L^{\infty}$--norm. 

We recall some basic facts about the viscous HJ equation \eqref{eq viscous HJ}. Let us begin with a definition.

\begin{definition}\label{def viscosity sol}
Let $v\in\cont(\T^d)$.
\begin{enumerate}[(i)]
 \item We will say that $v$ is a viscosity subsolution of \eqref{eq viscous HJ} if 
 \[
  -\Delta \varphi(x_0) +H\left(x_0,D\varphi(x_0)\right)\leq\alpha_0
 \]
 for every $\varphi\in\cont^2(\T^d)$ such that $v-\varphi$ has a local maximum at $x_0\in\T^d$. Such a function $\varphi$ will be called 
 {\em supertangent} to $v$ at $x_0$.\medskip
\item We will say that $v$ is a viscosity supersolution of \eqref{eq viscous HJ} if 
 \[
  -\Delta \varphi(x_0) +H\left(x_0,D\varphi(x_0)\right)\geq\alpha_0
 \]
 for every $\varphi\in\cont^2(\T^d)$ such that $u-\varphi$ has a local minimum at $x_0\in\T^d$. Such a function $\varphi$ will be called 
 {\em subtangent} to $v$ at $x_0$.\medskip
\end{enumerate}
We will say that $v$ is a solution if it is both a sub and a supersolution. 
\end{definition}

Solutions, subsolutions and supersolutions will be always assumed continuous 
in this paper and meant in the viscosity sense, hence the term 
{\em viscosity} will be omitted in the sequel.

\begin{oss}
One gets an equivalent definition of viscosity sub and super solution by replacing, in Definition \ref{def viscosity sol}, 
$\varphi\in\cont^2(\T^d)$ with $\varphi\in\cont^\infty(\T^d)$ and {\em local maximum} or {\em local minimum}) by {\em strict local maximum} or {\em strict local minimum}, see 
for instance \cite[Proposition 2.1]{barles_book}. 
\end{oss}

\begin{theorem}\label{teo eq viscous HJ}
Assume that $H$ satisfies {\bf (H1)},{\bf\,(H2)},{\bf\,(H3)}. 
  \begin{enumerate}[(i)]
  \item  Any Lipschitz viscosity solution $u$ of \eqref{eq viscous HJ} is of class
    $C^2$ and solve the equation in the classical sense.
  \item Classical solutions of \eqref{eq viscous HJ} are unique up to
    additive constants. 
    \item There is a unique real constant $\alpha_0$ for which the
    equation \eqref{eq viscous HJ} admits viscosity solutions.
  \item Any viscosity solution of \eqref{eq viscous HJ} is Lipschitz.
 \end{enumerate}
    \end{theorem}

It should be noted that, in Theorem \ref{teo eq viscous HJ} above, the assumption {\bf (H3)} is a rather strong requirement 
that is needed to conclude the uniqueness assertion (iii), but it is what we need in what follows.

    \begin{proof}[Sketch of the proof]\quad
      
      (i) Obviously, we have $-C\leq -\Delta u\leq C$ in the viscosity sense 
for some constant $C>0$, while from \cite{I} we have  $-C\leq -\Delta u\leq C$ in the viscosity sense 
if and only if $-C\leq -\Delta u\leq C$ in the distributional sense. Hence, 
$-\Delta u \in L^\infty(\T^d)$.  
Elliptic regularity theory ensures that $u\in W^{2,p}$ for any $p>1$
and, hence, $u\in \D{C}^{1,\sigma}$ for any $0<\sigma<1$.  Moreover, since
\[-\Delta u +H(x,Du)=\alpha\quad\hbox{ in } \T^d,\]
by the Schauder theory, we have $u\in \D{C}^{2,\sigma}$ for any $0<\sigma\leq \gamma$.

(ii) 
Let $u, v$ be classical solutions of \eqref{eq viscous HJ}. Pick $R\geqslant \|Du\|_\infty,\,  \|Dv\|_\infty$ and 
set $C:=\max_{\T^d\times B_{R+2}}|H(x,p)|$. By convexity, we have that $H(x,\cdot)$ is $C$--Lipschitz in $B_R$, for every $x\in\T^d$. 
Hence, by subtracting \eqref{eq viscous HJ} for $u$ and $v$, respectively,  we get
\[
0=-\Delta(u-v)+H(x,Du(x))-H(x,Dv(x))
\geq
-\Delta(u-v)-C|Du(x)-Dv(x)|\quad\hbox{in $\T^d$,}
\]
that is, $w:=u-v$ satisfies 
\[
-\Delta w -C|Dw|\leq 0 \ \ \text{ in }\T^d.
\] 
By the strong maximum principle, we infer that $w$ is a constant.

(iii) Observe first that the comparison result Theorem 3.3 in \cite{CIL} holds for the discounted equation
\begin{equation}
  \label{eq:discount}
  -\Delta u +H(x,Du)+\lam u =\alpha_0,\qquad\lam>0,
\end{equation}
and then use the argument in section II of \cite{LPV}.

(iv) The Lipschitz regularity is a consequence of Theorem VII.1 in \cite{IL}. 
Indeed, if we set 
\[
F(x,p,A)=-\tr A+H(x,p)-\alpha_0
\]
for every $(x,p)\in\R^{2d}$ and $d\times d$ real symmetric matrix $A$, then 
\[
|F(x,p,A)-F(y,p,A)|\leq Cd^{\gamma/2}
\]
due to the current assumption on $H$.  This ensures that $F$ satisfies (3.2) 
of \cite{IL}. The strict ellipticity (3.1) of \cite{IL} is valid with $F$, and thus 
\cite[Theorem VII.1]{IL} applies to \eqref{eq viscous HJ}. 

\end{proof}

\section{Discretization} \label{sez discretization}

Throughout this section we will assume $L:\T^d\times\R^d\to\R$ to satisfy condition {\bf (L1)},  {\bf (L2)}. We proceed to define a discrete operator  $\Leg_{\tau}:\D{C}(\T^d)\to\D{C}(\T^d)$, where the discretization parameter $\tau$ is taken in the interval $(0,1)$. 
Let us denote by $\parts(\T^d)$ the set of Borel probability measures on $\T^d$ endowed with the metrizable topology 
of weak*--convergence.  The source of randomness will be the heat kernel $\eta^\tau$  on $\T^d$, that is the continuous function
\begin{align*}
\eta^\tau:\T^d &\to \parts(\R^d)\\
y &\mapsto \eta^\tau_y,
\end{align*}
where $\eta^\tau_y$ is defined as follows:
\[
\eta^\tau_y(A):=\frac{1}{(4\pi\tau)^{\frac d2}}\int_{A} \e^{-\frac{|z-y|^2}{4\tau}}\,\dd z
 \qquad\hbox{for all $A\in\Bor(\R^d)$.} 
\]
Given $u\in\cont(\T^d)$, we have in particular 
\[
 \int_{\T^d} u(z)\,\dd\eta_y^\tau(z)=\frac{1}{(4\pi\tau)^{\frac d2}}\int_{\R^d} u(z)\,\e^{-\frac{|z-y|^2}{4\tau}}\,\dd z
 =
 (\eta^\tau*u)(y).
\]
where 
$\eta^\tau(y):=(4\pi\tau)^{-\frac d2}\e^{-\frac{|y|^2}{4\tau}}$. \smallskip

The discrete operator $\Leg_{\tau}:\D{C}(\T^d)\to\D{C}(\T^d)$ is defined as follows:  
\begin{equation}\label{def Leb tau}
\Leg_\tau  u(x) := \min_{q\in\R^d}\big( \tau L (x, -q) +(\eta^\tau*u)(x+\tau q)\big)\qquad\hbox{for every $x\in \R^d$.}  
\end{equation}

\begin{proposition}\label{prop Lax properties}
The operator $\Leg_\tau:\D{C}(\T^d)\to\D{C}(\T^d)$ is monotone and commute with additive constants, i.e. 
\begin{itemize}
\item[\em (i)] \quad $\cL_\tau u\le \cL_\tau v$ in $\R^d$\quad  if\quad  $u\leq v$ in $\R^d$;
\item[\em (ii)] \quad $\cL_\tau (u+k)=\cL_\tau u+k$ in $\T^d$ \quad for every $u\in\cont(\T^d)$ and $k\in\R$. 
\end{itemize}
In particular, 
$\|\cL_\tau u-\cL_\tau v\|_\infty\le\|u-v\|_\infty$.
\end{proposition} 
\begin{proof}
The fact that $\Leg_\tau$ is monotone and commutes with additive constants is apparent by its definition.  Since $v-\|u-v\|\le u\le v+\|u-v\|$, from items (i)--(ii) 
we infer 
\[\Leg_\tau v-\|u-v\|_\infty\le \Leg_\tau u\le \Leg_\tau v+\|u-v\|_\infty\qquad\hbox{in $\T^d$.}
\]
\end{proof}

The following holds:

\begin{prop}\label{prop a priori compactness}
Let $\tau>0$. Then there exists a constant $\kappa_\tau$ such that $\Leg_\tau u$ is $\kappa_\tau$--Lipschitz for every $u\in\D{C}(\T^d)$.
\end{prop}

\begin{proof}
For any fixed constant $A$, let us set  
\[
 Q_A(x):=\{q\in\R^d\,:\,L(x,-q)\leqslant A \}\qquad \hbox{for every $x\in\T^d$.}
\]
By the growth assumptions on $L$, there exist constants  $r(A), R(A)$ with  $\lim_{A\to +\infty} r(A) =  \lim_{A\to +\infty} R(A)= +\infty$ such that 
\[
 [0, r(A)]^d\subset Q_A(x) \subset [0, R(A)]^d.
\]
In particular, $Q_A(x)$ is a compact subset of $\R^d$ for every $x\in\T^d$. Choose $A_\tau$ large enough so that the set $Q_\tau(\cdot):=Q_{A_\tau}(\cdot)$ 
is such that 
\begin{equation}\label{eq Q_tau}
Q_\tau(x)+\frac{1}{\tau}\Z^d=\R^d\qquad\hbox{for every $x\in\T^d$}.
\end{equation}
Then, for every $q\not\in Q_\tau(x)$,  there exists $q_x\in Q_\tau(x)$ such that $q-q_x=\Bbbk/\tau$ for some $\Bbbk\in\Z^d$, i.e. $\tau q=\tau q_x+\Bbbk$. Then, given 
$u\in\D{C}(\T^d)$, by periodicity we get
\[
 (\eta^\tau*u)(x+\tau q)=(\eta^\tau*u)(x+\tau q_x+\Bbbk)=(\eta^\tau*u)(x+\tau q_x)
\]
while $L(x,-q)> A_\tau\geqslant L(x,q_x)$. Setting $R_\tau:=R_{A_\tau}$, we have in particular   
\[
 \Leg_\tau u (x) = \min_{q\in [0, R_\tau]^d}\big( \tau L(x,-q) +(\eta^\tau*u)(x+\tau q)\big)\qquad\hbox{for all $x\in\T^d$.}
\]
Let us denote by $K_\tau$ a Lipschitz constant of $L$ on $\T^d\times \left[-\frac 1\tau, R_\tau+\frac 1\tau \right]^d$. Let 
$u\in\cont(\T^d)$ and pick $x_1,x_2\in\T^d$. Let us denote by $q\in [0, R_\tau]^d$ a minimizing vector for $\Leg_\tau u(x_2)$ and set $\xi:=q+(x_2-x_1)/\tau$ so that 
$x_2+\tau q= x_1 +\tau \xi$. We have
\begin{eqnarray*}
\Leg_\tau u (x_1)-\Leg_\tau u(x_2) 
&\leqslant 
\tau L(x_1,-\xi) +(\eta^\tau*u)(x_1+\tau \xi)-\tau L(x_2,-q)-(\eta^\tau*u)(x_2+\tau q)\\
&= \tau \left(L(x_1,-\xi)-L(x_2,-q)\right) \leqslant K_\tau (1+{\tau})\,|x_1-x_2|.
\end{eqnarray*}
This gives the assertion with $\kappa_\tau:=K_\tau (1+{\tau})$. 
\end{proof}

We end this section with a result we will need in the sequel.

\begin{prop}\label{prop argmin L tau}
For $u\in\cont(\T^d)$, $\tau>0$ and $x\in\R^d$ we set 
\[
 \D{argmin}(\Leg_\tau u(x)):=\left\{q\in\R^d\,\mid\,\Leg_\tau u(x)=\tau L\left(x,-q\right)+(\eta^\tau * u)(x+\tau q)\right\}
\]
Then
\[
 -q\in\partial_p H\left( {x}, D(\eta^\tau*u)(x+\tau q)\right)
 \qquad\hbox{for all $q\in \D{argmin}(\Leg_\tau u(x))$.}
\]
\end{prop}

\begin{proof}
Let us fix $x\in\T^d$. Pick a $\hat q\in \D{argmin}(\Leg_\tau u(x))$. Then the function $q\mapsto \tau L(x,-q)+(\eta^\tau * u)(x+\tau q)$ has a minimum at $\hat q$. This 
implies
\[
 0\in -\partial_q L(x,-q)+D(\eta^\tau * u)(x+\tau q),
\]
or, otherwise stated, 
\[
D(\eta^\tau * u)(x+\tau q)\in \partial_q L(x,-q).
\]
The assertion follows by convex duality.
\end{proof}

%
%
We are  interested in finding solutions of the following identity
\begin{equation}\label{eq fixed points} 
\Leg_\tau  u = u -\tau\alpha\qquad\hbox{in $\T^d$,} 
\end{equation} 
where $\alpha\in\R$ and $u\in\cont(\T^d)$. We start with the following uniqueness result:

\begin{theorem}\label{teo unique alpha tau}
There exists at most one constant $\alpha\in\R$ for which equation \eqref{eq fixed points} admits solutions in $\cont(\T^d)$. 
Furthermore, the solution $u\in\cont(\T^d)$ of \eqref{eq fixed points} is unique, up to additive constants. 
\end{theorem}

\begin{proof}
Let $u_1 $ and $u_2$ fixed points with constants $\alpha_1$ and $\alpha_2$, respectively.  Let $x$  be a maximum of the difference $u_1- u_2$. Let $q\in\R^d$
 such that 
 $$ 
 u_2 (x) = \tau L\left(x, q \right) + \int_{\T^d}u_2(z)\, d\eta^\tau_{y} (z) +\tau\alpha_2. 
$$
By definition we have 
$$ 
u_1 (x) \leq  \tau L\left(x,q\right) + \int_{\T^d}u_1(z)\, d\eta^\tau_{y} (z) +\tau\alpha_1. 
$$
So 
$$
u_1 (x) - u_2 (x)  \leq \int_{\T^d}   \big( u_1 (z) - u_2 (z)\big)\, d\eta_{y}(z)  +\tau(\alpha_1 -\alpha_2). 
$$
Since $x$ is a maximum of $u_1-u_2$, we get $\tau(\alpha_1-\alpha_2)\geq 0$, hence  $\alpha_1\geq \alpha_2$.  By symmetry, we obtain the equality.

Let us now assume $u_1, u_2$ solutions to \eqref{eq fixed points} for the same $\alpha $.  By arguing as above 
we get
$$
u_1 (x) - u_2 (x)  \leq \int_{\T^d}   \big(u_1 (z) - u_2 (z)\big)\, d\eta^\tau_{y}(z)  \leq \max_{\T^d}\big(u_1-u_2)= u_1 (x) - u_2 (x), 
$$
hence $u_1 (z) - u_2 (z)   = u_1 (x) - u_2 (x)$ for every $z\in\D{spt}\left(\eta^\tau_{y}\right)=\T^d$.
\end{proof}

Let us proceed to show existence.

\begin{proof}[Proof of Theorem \ref{teo existence alpha tau}]
Let us denote by $\widehat{\D{C}}(\T^d)$ the quotient space of $\D{C}(\T^d)$, where we identify functions that differ 
by a constant, and by  $q:\D{C}(\T^d)\to\widehat{\D{C}}(\T^d)$ the projection. 
Since $\Leg_\tau$ commutes with the addition of constants, it defines an operator
$\hcL_\tau:\widehat{\D{C}}(\T^d)\to \widehat{\D{C}}(\T^d)$. 
Let us denote by $\Lip_{\kappa_\tau}(\T^d)$ the family of $\kappa_\tau$--Lipschitz function on $\T^d$, where $\kappa_\tau$ is 
the constant provided by Proposition \ref{prop a priori compactness}. 
The set $\widehat{\Lip}_{\kappa_\tau}(\T^d):=q\left(\Lip_{\kappa_\tau}(\T^d)\right)$ is a convex and compact subset of 
$\widehat{\D{C}}(\T^d)$, so we can apply Schauder fixed point Theorem (see for instance \cite[Theorem 3.2, p. 415]{Dug}) to infer that the 
operator $\hcL_\tau:\widehat{\D{C}}(\T^d)\to \widehat{\D{C}}(\T^d)$ has a fixed-point $\hat{u}_\tau\in \widehat{\Lip}_{\kappa_\tau}(\T^d)$, 
i.e. $\hcL_\tau(\hat{u}_\tau)=\hat{u}_\tau$. Lifting these relations to $\D{C}(\T^d)$, we infer that there exists a constant $\alpha_\tau\in\R$ 
such that $\Leg_\tau u_\tau= u_\tau-\tau\alpha_\tau$ \ \ in $\T^d$ with $u_\tau=q^{-1}(\hat{u}_\tau)\in \Lip_{\kappa_\tau}(\T^d)$. 
{The asserted uniqueness of $\alpha_\tau$ in $\R$ and $u_\tau$ in $\cont(\T^d)$   
is guaranteed by Theorem \ref{teo unique alpha tau}.} 
\end{proof}

In view that discretization is often associated with numerical computations, we give another proof of Theorem \ref{teo existence alpha tau}, 
which relies on Banach's fixed point theorem instead of  Schauder's fixed point theorem.  

\def\gd{\delta}

\begin{proof}[Second proof of Theorem \ref{teo existence alpha tau}] Let $\gd>0$, and consider the problem
$(1+\gd)u-\cL_\tau u=0$ in $\T^d$. By Proposition 2.1, $\cL_\tau : \D{C}(\T^d) \to \D{C}(\T^d)$ is $1$-Lipschitz. Hence, by Banach's fixed point theorem, $(1+\gd)^{-1}\cL_\tau$ has a unique fixed point $v^\gd\in \D{C}(\T^d)$, which is a unique solution of 
$(1+\gd)u-\cL_\tau u=0$ in $\T^d$. Let $\kappa_\tau>0$ be the constant given by Proposition 2.2, so that $v^\gd=(1+\gd)^{-1}\cL_\tau v^\gd$ is $(1+\gd)^{-1}\kappa_\tau$-Lipschitz on $\T^n$. Accordingly, the family $\{v^\gd \,|\, \gd>0\}$ is equi-Lipschitz on $\T^d$. By the Ascoli-Arzela theorem, we can select a sequence of $\gd_j>0$ converging to zero such that the functions $v^{\gd_j}-\min_{\T^d}v^{\gd_j}$ converge 
to a function $w$ in $\D{C}(\T^d)$ as $j\to\infty$.  Setting 
$m_j:=\min_{\T^d}v^{\gd_j}$and $w_j:=v^{\gd_j}-m_j$, we observe by Proposition 2.1 that
\[
0=(1+\gd_j) (w_j+m_j)-\cL_\tau(w_j+m_j)
= (1+\gd_j)w_j+\gd_j m_j-\cL_\tau w_j, 
\]
where the first and last terms in the last expression converge to 
$w$ and $\cL_\tau w$ in $\D{C}(\T^d)$, respectively. Consequently, 
the sequence of the constants $\gd_j m_j$ converges to a constant $-\tau\alpha_\tau$, which implies that $\cL_\tau w-w=\tau\alpha_\tau$ in $\T^d$. 
\end{proof} 

The standard proof of Banach's fixed point theorem is constructive or iterative, and therefore, the above proof can be easily implemented for numerical computations.

\begin{definition}
We say that $u\in\cont(\T^d)$ is an $\alpha$-subsolution for $\Leg_\tau$ if 
$$
\cL_\tau u\geq u-\tau\alpha\qquad\hbox{in $\T^d$.}
$$
Denote by $\cH_\tau(\af)$ the set of $\af$-subsolutions.
\end{definition}

By taking into account the properties of the Lax operator stated in Proposition \ref{prop Lax properties}, we easily infer the following facts:

\begin{prop}\label{prop properties L tau}
The sets $\cH_\tau(\alpha)$ are convex and closed subset of $\D C(\T^d)$ and increasing with respect to $\alpha$, i.e. $\cH_\tau(\af)\subseteq\cH_\tau(\be)$ 
if $\alpha\leq\beta$. Furthermore:
\begin{itemize}
\item[\em (i)] $u+k\in\cH_\tau(\alpha)$ for every $u\in\cH_\tau(\alpha)$ and $k\in\R$;\medskip
\item[\em (ii)] $\Leg_\tau\left(\cH_\tau(\alpha)\right)\subseteq \cH_\tau(\alpha)$. 
\end{itemize}
\end{prop}

Next, we show that all $\alpha_\tau$--subsolutions for $\Leg_\tau$ are actually solutions to \eqref{eq discrete HJ}. 

\begin{proposition}\label{prop sub=sol}
Let $(\alpha_\tau,u_\tau)\in\R\times\D{C}(\T^d)$ be a solution of 
\begin{equation*}
\Leg_\tau u_\tau=u_\tau-\tau \alpha_\tau\qquad\hbox{in $\T^d$}.
\end{equation*}
Then   $\cH_\tau(\alpha_\tau)=\{u_\tau+k\,:\,k\in\R\,\}$. Furthermore, 
\begin{equation}\label{def alpha tau}
\alpha_\tau=\min\{\af:\cH_\tau(\af)\ne\emptyset\}.
\end{equation}
\end{proposition}

\begin{proof}
Let us pick $u\in\cH_\tau(\alpha)$ and argue as in the proof of Theorem \ref{teo unique alpha tau} with $u_1:=u$ and $u_2:=u_\tau$. By also using the fact that 
$u\leq \Leg_\tau u +\tau\alpha$, we end up with 
\[
u(x)-u_\tau(x)\le \int_{\T^d} (u(z)-u_\tau(z))d\eta_y(z)+\tau(\alpha-\alpha_\tau)\qquad\hbox{for all $x\in\T^d$.}
\]
By picking as $x$ a maximum point of $u-u_\tau$, we conclude that $\alpha\geq\alpha_\tau$. When $\alpha=\alpha_\tau$, we furthermore get that $u-u_\tau$ is constant.   
\end{proof}

We conclude this section by deriving the following bounds on the constant $\alpha_\tau$. 

\begin{prop}\label{prop alpha_tau}
The following holds: 
\begin{equation*}\label{eq alpha tau}
-\max_{y\in\T^d}\min_{q\in\R^d} L(y,q) \leq \alpha_\tau \leq -\min_{\T^d\times\R^d} L.
\end{equation*}
\end{prop}

\begin{proof}
Let us set $L_m(x):=\min_{q\in\R^{d}} L(x,q)$. 
Pick $u\in \cH_\tau(\af_\tau)$ and 
set $v(x):=\max_{\T^d} u$,\ $w(x):=\min_{\T^d} u$ for all $x\in\R^d$. Then, for all $x\in\T^d$,   
\[
u(x)-\tau\alpha_\tau 
=
\cL_\tau u(x)
\leq 
\cL_\tau v(x)
=  \tau L_m(x) +\max_{\T^d} u
\leq \tau \max_{\T^d} L_m+\max_{\T^d} u \]
and 
\[
u(x)-\tau\alpha_\tau 
=
\cL_\tau u(x)
\geq 
\cL_\tau w(x)
=
\tau L_m(x) +\min_{\T^d} u
\geq 
\tau \min_{\T^d} L_m+\min_{\T^d} u. 
\]
That implies $-\max_{\T^d} L_m \le \alpha_\tau \leq -\min_{\T^d} L_m$, as it was asserted.
\end{proof}
\ 

\section{Equi-continuity of the functions $u_\tau$}\label{sez equicontinuity}

This section is devoted to prove equi-continuity of the functions 
$\{u_\tau\,:\, \tau\in (0,1)\,\}$, where $u_\tau$ denotes a solution in 
$C(\T^d)$ of the equation 
\begin{equation} \label{eq u tau}
\cL_\tau u_\tau=u_\tau-\tau\alpha_\tau \ \ \text{ in } \T^d
\end{equation}
and $\alpha_\tau$ is the constant given by Theorem \ref{teo existence alpha tau}. 
Throughout the rest of the paper, we will assume that $L$ satisfies the following further condition, for some constants $\gamma\in(0,\,1)$ and $M_\gamma >0$:
\begin{equation}\label{L3}
L(x+h,q+k)+L(x-h,q-k)-2L(x,q)\leq M_\gamma (|h|^\gamma+|k|) \tag{\bf L3}
\end{equation}
for all $(x,q)\in\T^d\times\R^d$ and $h,k\in B_{R_d}$ with $R_d:=\sqrt{d}/2$. 
\def\cF{\mathcal{F}}
\def\cE{\mathcal{E}}
\def\gT{\Theta} 
\def\:{\,\operatorname{:}\,}

We start by noticing that  
\[
\cL_\tau u(x)
=\min_{q\in\R^d}\left(\tau L\Big(x,\frac q\tau\Big)+(\eta^\tau*u)(x-q)\right)
=\min_{y\in\R^d}\left(\tau L\Big(x,\frac{x-y}\tau\Big)+(\eta^\tau*u)(y)\right).
\]
We introduce the operators $\cE_\tau, \cF_\tau \: \D{C}(\T^d)\to \D{C}(T^d)$ defined as 
\[
\cE_\tau u=\eta^\tau* u,\qquad
\cF_\tau u(x)=\min_{y\in\R^d}\left(\tau L\Big(x,\frac{x-y}\tau\Big)+u(y)\right). 
\]
Notice that \ \ $\cL_\tau=\cF_\tau \circ \cE_\tau.$
For $u\in \D{C}(\T^d)$ and $\sigma\in(0,\,2]$, we set 
\[
\Theta_\sigma(u):=\inf\{a\geq 0 \mid u(x+h)+u(x-h)-2u(x)\leq a|h|^\sigma \ \text{ for all } x,h\in\R^d\}, 
\]
and 
\[
\Lambda_\sigma(\T^d)=\{v\in \D{C}(\T^d)\mid \Theta_\sigma(v)<\infty\}.
\]
We also introduced the following temporary notation, defined for $R>0$:
\[
\Theta_{\sigma,R}(u):=\inf\{a\geq 0\mid u(x+h)+u(x-h)-2u(x)
\leq a|h|^\sigma \ \ \text{ for } x, h\in\R^d, \ \text{ with }|h|\leq R\}.
\]
We start with some preliminary results.

\begin{lemma} \label{with R} 
Let $u\in \Lambda_\sigma(\T^d)$, with $\sigma\in(0,\,2]$, and $R\geq R_d:=\sqrt d/2$. Then 
\[
\Theta_{\sigma,R}(u)=\Theta_\sigma(u). 
\]
\end{lemma}

\begin{proof}  It is obvious that
\begin{equation} \label{with R+1}
\Theta_{\sigma,R}(u){ \leq} \Theta_\sigma(u). 
\end{equation}
To prove the reversed inequality, we fix $x\in\R^d$ and $a\geq 0$ and assume that
\begin{equation}\label{with R+2}
u(x+h)+u(x-h)-2u(x)\leq a|h|^\sigma \ \ 
\text{ for } x, h\in\R^d, \ \text{ with }|h|\leq R. 
\end{equation}
Set 
\[
f(h):=u(x+h)+u(x-h)-2u(x) \ \ \text{ for all }h\in\R^d,
\]
and observe that $f\in \D{C}(\T^d)$. 
By the periodicity of $f$, we see that
\[
M :=\max_{\R^d} f=\max\{f(h)\mid h\in\R^d,\, |h_i|\leq 1/2 \ \ \text{ for } i=1,\ldots,d\}. 
\]
Since $f(0)=0$, we have $M \geq 0$. 

By \eqref{with R+2}, we have 
\[\begin{aligned}
M &\,
\leq a \max\{|h|^\sigma\mid h\in\R^d,\, |h_i|\leq 1/2 \ \text{ for } i=1,\ldots,d\}
\\&\,\leq a \max\{|h|^\sigma\mid h\in\R^d,\, |h|\leq R\}=aR^\sigma,
\end{aligned}
\]
which implies that 
\[
f(h)\leq M \leq aR^\sigma\leq a|h|^\sigma \ \ \ \text{ for }\ h\in \R^d\setminus B_R. 
\]
This together with \eqref{with R+2} yields 
\[
f(h)\leq a|h|^\sigma \ \ \ \text{ for all } h\in \R^d.  
\]
Thus, we have the reversed inequality of \eqref{with R+1} and conclude that 
\ $
\Theta_\sigma(u)=\Theta_{\sigma,R}(u). 
$
\end{proof} 

We derive the following consequence. 

\begin{corollary} \label{inclusion} Let $0<\rho<\sigma\leq 2$. Then 
\[
\Lambda_\sigma(\T^d)\subset \Lambda_\rho(\T^d).
\]
\end{corollary}

\begin{proof} Let $u\in\Lambda_\sigma(\T^d)$ and $x,h\in\R^d$. If $|h|\leq R_d$,  
then we have 
\[
u(x+h)+u(x-h)-2u(x)\leq \Theta_\sigma(u)|h|^\sigma\leq \Theta_\sigma(u)R_d^{\sigma-\rho}|h|^\rho,
\]
which shows that $\Theta_{\rho,R_d}(u)<\infty$.  By Lemma \ref{with R}, we see that 
$\Theta_\rho(u)<\infty$ and, moreover, that $\Lambda_\sigma(\T^d)\subset \Lambda_\rho(\T^d)$. 
\end{proof}

For {$\sigma\in (0,\,1]$} and $u\in \D{C}(\R^d)$, we write 
\[
\Lip_\sigma(u)=\sup_{x,y\in\R^d, \ x\not=y}\frac{|u(x)-u(y)|}{|x-y|^\sigma},
\]
and 
\[
\D{C}^{0,\sigma}(\T^d):=\Big\{u\in \D{C}(\T^d)\mid \Lip_\sigma(u)<\infty\Big\}. 
\]

The following lemma is similar to Corollary \ref{inclusion}. 

\begin{lemma} \label{inclusion2} Let $0<\rho<\sigma\leq 1$. Then 
\[
 \D{C}^{0,\sigma}(\T^d)\subset\D{C}^{0,\rho}(\T^d). 
\]
\end{lemma}

\begin{proof} Let $u\in \D{C}^{0,\sigma}(\T^d)$. For any $x,y\in\R^d$,
we choose $z\in\Z^d$ so that 
\[
|x-y-z|=\min_{\zeta\in\Z^d}|x-y-\zeta|,
\]
and note that 
\[
z\in\prod_{i=1}^d [x_i-y_i-1/2,\, x_i-y_i+1/2],
\]
and 
\[\begin{aligned}
|u(x)-u(y)|&\,=|u(x)-u(y+z)|\leq \Lip_\sigma(u)|x-y-z|^\sigma
\\&\,\leq  \Lip_\sigma(u)R_d^{\sigma-\rho}|x-y-z|^\rho
\leq  \Lip_\sigma(u)R_d^{\sigma-\rho}|x-y|^\rho.
\end{aligned}
\]
This shows that \ $ \D{C}^{0,\sigma}(\T^d)\subset\D{C}^{0,\rho}(\T^d)$.  
 \end{proof} 

\begin{proposition} \label{prop-equiv} For any $\sigma\in(0,\,1)$, we have
\[
\Lambda_\sigma(\T^d)= \D{C}^{0,\sigma}(\T^d)
\]
and for some constant $C_\sigma>1$, depending only on $\sigma$, 
\[
C_\sigma^{-1}\Theta_\sigma(u)\leq \Lip_\sigma(u)\leq C_\sigma \Theta_\sigma(u) \ \ \text{ for } u\in \Lambda_\sigma(\T^d). 
\]
\end{proposition}

\begin{proof} Let $u\in \D{C}^{0,\sigma}(\T^d)$ and $x,h\in\R^d$. Compute that
\[\begin{aligned}
u(x+h)+u(x-h)-2u(x)&\,=u(x+h)-u(x)+u(x-h)-u(x)
\\&\,\leq \Lip_\sigma(u)|h|^\sigma+\Lip_\sigma(u)|h|^\sigma=2\Lip_\sigma|h|^\sigma,
\end{aligned}
\]
which shows that 
\[
\Theta_\sigma(u)\leq 2\Lip_\sigma(u),
\]
and that $ \D{C}^{0,\sigma}(\T^d)\subset \Lambda_\sigma(\T^d)$. 

Now, let $u\in\Lambda_\sigma(\T^d)$, so that we have 
\begin{equation} \label{gs<1+1}
u(x)\geq \frac 12(u(x+h)+u(x-h)-\Theta_\sigma(u)|h|^\sigma) \ \ \text{ for all } x,h\in\R^d.
\end{equation}

Let $x,y\in\R^d$. We intend to show that 
\[
u(x)-u(y)\leq C|x-y|^\sigma
\]
for some constant $C>0$, depending only on $\Theta_\sigma(u)$ and $\sigma$. 

By translation, we may assume that $x=0$. 
We need to show that for some constant $C>0$, 
\begin{equation}\label{equiv1}
u(y)\geq u(0)-C|y|^\sigma.  
\end{equation}

We may assume that $y\not=0$. 
Set $K=\Theta_\sigma(u)$ and let $m\in\N$ large enough.  
By \eqref{gs<1+1}, we have 
\[
u(2^{k-1}y)\geq \frac 12(u(0)+u(2^ky)-K|2^{k-1}y|^\sigma) \ \ \text{ for } k=1,\ldots,m.
\]
From these, we obtain
\[\begin{aligned}
\sum_{k=1}^m 2^{-(k-1)}u(2^{k-1}y)&\,\geq 
\sum_{k=1}^m 2^{-k}\big[u(0)+u(2^{k}y)-K 2^{\sigma(k-1)}|y|^\sigma\big]
\\&\,=\frac{2^m-1}{2^m}u(0)+\sum_{k=1}^m 2^{-k}u(2^{k}y)
-K2^{-\sigma} \sum_{k=1}^m2^{(\sigma-1)k}|y|^\sigma.
\end{aligned}
\]
{After rewriting the first summation above  as $u(y)+\sum_{k=1}^{m-1} 2^{-k}u(2^{k}y)$, we get}
\[
\begin{aligned}
u(y)&\,\geq (1-2^{-m})u(0)+2^{-m}u(2^my)-K2^{-\sigma}2^{\sigma-1}\frac{1-2^{(\sigma-1)m}}{1-2^{\sigma-1}}\,|y|^\sigma
\\&\,\geq u(0)+2^{-m}(u(2^my)-u(0))-\frac{K}{2-2^{\sigma}}\,|y|^\sigma
\\&\,\geq u(0)-2^{-m}\operatorname{osc}(u)-\frac{K}{2-2^{\sigma}}\,|y|^\sigma,
\end{aligned}
\]
where $\operatorname{osc}(u):=\max_{\T^d} u-\min_{\T^d} u$.

Sending $m\to +\infty$ yields 
\[
u(y)\geq u(0)-\frac{\Theta_\sigma(u)}{2-2^\sigma}|y|^\sigma,
\]
which proves \eqref{equiv1}, with $C=\Theta_\sigma(u)/(2-2^\sigma)$. This readily shows that 
\[
|u(x)-u(y)|\leq \frac{\Theta_\sigma(u)}{2-2^\sigma}|x-y|^\sigma \ \ \text{ for all } x,y \in\R^d.
\]
Hence, we have  
\[
\Lip_\sigma(u)\leq \frac{1}{2-2^\sigma}\Theta_\sigma(u) \ \ \text{ and hence } \ \ \Lambda_\sigma(\T^d)\subset  \D{C}^{0,\sigma}(\T^d).
\]
The proof is now complete. 
\end{proof}

One can show that if $\sigma\in(1,\,2]$, then $\Lambda_\sigma(\T^d)\subset \D{C}^{0,1}(\T^d)$, 
which is left to the interested reader to check.

Let us now prove the equi--continuity of the functions $\{u_\tau\,:\, \tau\in (0,1)\,\}$. 
We start with the following result:

\begin{theorem} \label{main-ineq} 
Let $u\in\Lambda_\gamma(\T^d)$. 
Then
\[
\Theta_\gamma(\cF_\tau u)\leq \frac{1}{(1+\tau)^\gamma }\Theta_\gamma(u)+\tau M_\gamma
\left(1+\frac{R_d^{1-\gamma}}{1+\tau}\right). 
\]
\end{theorem}

\begin{proof}
Let $x\in \R^d$ and choose $y\in\R^d$ so that 
\[
\cF_\tau u(x)=u(y)+\tau L\Big(x,\frac{x-y}{\tau}\Big).
\]
For every $h,k\in\R^d$ we get 
\[ \begin{aligned}
&\cF_\tau u(x+h)+\cF_\tau u(x-h)-2\cF_\tau u(x)
\\&\leq u(y+k)+\tau L\Big(x+h, \frac{x+h-(y+k)}{\tau}\Big)
+u(y-k)+\tau L\Big(x-h,\frac{x-h-(y-k)}\tau\Big)
\\&\quad -2u(y)-2\tau L\Big(x,\frac{x-y}{\tau}\Big)
\\&
\leq
\Theta_\gamma(u)|k|^\gamma 
+
\tau\left[
L\Big(x+h, \frac{x-y}{\tau}+\frac{h-k}{\tau}\Big)
+
L\Big(x-h, \frac{x-y}{\tau}-\frac{h-k}{\tau}\Big)
-
2\tau L\Big(x,\frac{x-y}{\tau}\Big)
\right]
\end{aligned}
\]
In order to exploit \eqref{L3}, we take $h\in \overline B_{R_d}$ and $k:=h/(1+\tau)$, so that $\left|\frac{h-k}{\tau}\right|=\left|\frac{h}{1+\tau}\right|<R_d$. 
We infer 
\[ \begin{aligned}
&\cF_\tau u(x+h)+\cF_\tau u(x-h)-2\cF_\tau u(x)
\leq 
\Theta_\gamma(u)\frac{|h|^\gamma}{(1+\tau)^\gamma}+ \tau M_\gamma
 \left(|h|^\gamma +\frac{|h|}{1+\tau}\right),
\end{aligned}
\] 
which implies
\[
\Theta_{\gamma,R_d}(\cF_\tau u)
\leq 
\frac 1{(1+\tau)^\gamma}\Theta_\gamma(u)
+\tau M_\gamma\left(1+\frac{R_d^{1-\gamma}}{1+\tau}
\right). 
\]
The assertion follows in view of Lemma \ref{with R}. 
\end{proof}

\begin{lemma} \label{heat} Let $u\in\Lambda_\sigma(\T^d)$, with $\sigma\in(0,\,2]$.  We have 
\[
\Theta_\sigma(\cE_\tau u)\leq \Theta_\sigma(u). 
\]
\end{lemma}

\begin{proof} Let $x,\,h\in\R^d$. 
We compute that
\[\begin{aligned}
&\cE_\tau u(x+h)+\cE_\tau u(x-h)-2 \cE_\tau u(x)
\\&=\int_{\R^d} \eta^\tau(y)\Big(u(x-y +h)+u(x-y-h)-2u(x-y)\Big) dy
\\&\leq  \int_{\R^d}\eta^\tau(z)\Theta_\sigma(u) |h|^\sigma dz
=\Theta_\sigma(u)|h|^\sigma,
\end{aligned}
\]
which yields
\[
\Theta_\sigma(\cE_\tau u)\leq \Theta_\sigma(u). 
\]
\end{proof}

\begin{theorem} \label{cL} 
Let $u\in\Lambda_\gamma(\T^d)$. Then
\[
\Theta_\gamma(\cL_\tau u)\leq \frac{1}{(1+\tau)^\gamma }\Theta_\gamma(u)+\tau M_\gamma
\left(1+\frac{R_d^{1-\gamma}}{1+\tau}
 \right). 
\]
\end{theorem}

\begin{proof} By Theorem \ref{main-ineq} and Lemma \ref{heat}, we obtain 
\begin{align*}
\Theta_\gamma(\cL_\tau u)
&\,=\Theta_\gamma(\cF_\tau\circ \cE_\tau u)
\\&\,\leq\frac{1}{(1+\tau)^\gamma }\Theta_\gamma(\cE_\tau u)+\tau M_\gamma
\left(1+\frac{R_d^{1-\gamma}}{1+\tau} \right)
\\&\,\leq \frac{1}{(1+\tau)^\gamma }\Theta_\gamma(u)+\tau M_\gamma
\left(1+\frac{R_d^{1-\gamma}}{1+\tau} \right). 
\end{align*}
\end{proof}

%
%

\begin{lemma} \label{regularizing} Let $u\in \D{C}(\T^d)$ and $\sigma\in(0,\,2]$. Then
\[
\cE_\tau u\in\Lambda_\sigma(\T^d).
\]
\end{lemma}

\begin{proof} Set 
\[
v_\tau(x)=\cE_\tau u(x) \ \ \text{ for } x\in\T^d.
\]
As is well-known (and easily shown), the function $v_\tau$ is smooth and periodic in $\R^d$. 
In particular, the second derivatives of $v_\tau$ are bounded in $\R^d$, which 
implies that $v_\tau$ is semi-concave in $\R^d$, that is, $\Theta_2(\cE_\tau u)<\infty$.  
Thus, we find that $\cE_\tau u\in\Lambda_2(\T^d)$ and, due to Corollary \ref{inclusion},  
that $\cE_\tau u\in\Lambda_\sigma(\T^d)$ for all $\sigma\in(0,\,2]$. 
\end{proof}

\begin{theorem} \label{main-est} 
Let $\tau>0$ and $u_\tau\in \D{C}(\T^d)$ satisfy \eqref{eq u tau}. 
Then 
\[
\Theta_\gamma (u_\tau)\leq \frac{\tau B_\tau}{1-A_\tau},
\]
where 
\[
A_\tau=\frac{1}{(1+\tau)^\gamma},\quad 
B_\tau=M_\gamma\left(1+\frac{R_d^{1-\gamma}}{1+\tau} \right) \ \ \text{ and } \ \ R_d=\frac{\sqrt d}{2}.
\]
\end{theorem}

We remark that, in the theorem above, \ $0<A_\tau<1$, 
\[
\lim_{\tau\to 0}\frac{\tau}{1-A_\tau}=\frac{1}{\gamma},
\]
and for any $0<T<\infty$, 
\[
\sup_{0<\tau\leq T}\frac{\tau B_\tau}{1-A_\tau}<\infty.
\]

\begin{proof} Using Lemma \ref{regularizing} and Theorem \ref{main-ineq}, we infer 
from \eqref{eq u tau} that 
\[
u_\tau=\cF_\tau\circ \cE_\tau u_\tau+\tau\alpha_\tau\in \Lambda_\gamma(\T^d).
\]
By \eqref{eq u tau} and Theorem \ref{cL}, we get 
\[
\Theta_\gamma(u_\tau)=\Theta_\gamma(u_\tau-\tau \alpha_\tau)
=\Theta_\gamma(\cL_\tau u_\tau)
\leq A_\tau \Theta_\gamma(u_\tau)+\tau B_\tau,
\]
from which follows
\[
\Theta_\gamma(u_\tau)\leq \frac{\tau B_\tau}{1-A_\tau}.  
\]
\end{proof}

%
%
%
%
%
%
%

As a consequence of the information gathered, we derive the following fact:

\begin{prop}\label{equi-con}
{The family of functions $\{u_\tau\mid \tau\in (0,\,1)\}$ is equi-continuous on $\T^d$.}  
\end{prop}
  \begin{proof}
    It is a direct consequence of Theorem \ref{main-est}  and Proposition \ref{prop-equiv} 
  \end{proof}

\section{The approximation result}
\subsection{Proof of Theorem \ref{teo asymptotic}.} \label{sec:proof-main-result}
This section is devoted to the proof of Theorem \ref{teo asymptotic}. We begin by showing that, under assumptions {\bf (L1)}, {\bf (L2)}, {\bf (L3)} on 
the Lagrangian $L$, Theorem \ref{teo eq viscous HJ} applies.

\begin{proposition}\label{H-Holder} Assume that $L$ satisfies {\bf (L1)},\,{\bf (L2)},\,\eqref{L3}. 
Then the associated Hamiltonian $H$ satisfies conditions  {\bf (H1)},{\bf\,(H2)},{\bf\,(H3)}.
\end{proposition}

\begin{proof}
The fact that $H$ satisfies {\bf (H1)},{\bf\,(H2)}, i.e. it is convex and superlinear, is standard. Let us prove {\bf (H3)}.  
For each $q\in\R^d$ the function $u_q(x)= L(x,q)$ belongs to $\Lam_\ga(\T^d)$ with
$\Theta_\ga(u_q)\le M_\gamma $. From Proposition \ref{prop-equiv} we get $u_q\in\D{C}^{0,\gamma}(\T^d)$ and
$\Lip_\ga(u_q)\le C_\ga M_\gamma :=D_\ga$. Thus
\[|L(x,q)-L(y,q)|\le D_\ga |x-y|^\ga \hbox{ for any }
x,y\in\T^d, q\in\R^d.\]
For $x,p,h\in\R^d$, let $q_\pm$ be such that
\[H(x\pm h,p)=pq_\pm-L(x,q_\pm)\]
and let $v=\frac 12(q_++q_-)$, $u=\frac 12(q_+-q_-)$, so that $q_\pm=v\pm u$.
Then
\begin{align*}
&  H(x+h,p)+H(x {-}h,p)-2H(x,p)\\
&\le p(v+u)+p(v-u)-2pv-L(x+h,v+u)-L(x-h,v-u)+2L(x,v)\\
  &=L(x,v+u)-L(x+h,v+u)+L(x,v-u)-L(x-h,v-u)\\
  &-L(x,v+u)-L(x,v-u)+2L(x,v)\\
&\le 2D_\ga|h|^\ga,
\end{align*}
where, for the last inequality, we have also exploited the convexity of $L(x,\cdot)$. 
Thus, for each $p\in\R^d$, the function $w_p(x)=H(x,p)$ belongs to
$\Lam_\ga(\T^d)$ with $\gT_\ga(w_p)\le 2D_\ga$, so
we have that $w_p\in\D{C}^{0,\gamma}(\T^d)$ with
$\Lip_\ga(w_p)\le 2 C_\ga D_\ga=2M_\gamma C_\ga^2$.
\end{proof}

Next, we prove an auxiliary lemma.

\begin{lemma}\label{lemma tool}
Let $\varphi\in\cont^{2}(\T^d)$. For every $R>0$, there exists a 
continuous function $\omega:[0,+\infty)\to [0,+\infty)$ vanishing at $0$, only depending on $R$ and $\varphi$, such that 
\[
\left|\frac{(\eta^\tau*\varphi)(x+\tau q) - \varphi(x)}{\tau}-\langle D\varphi(x),q\rangle -\Delta\varphi(x)\right|
\leq \omega(\tau)
\]
for all $(x,q)\in \T^d\times B_R$ and $\tau>0$.
\end{lemma}

\def\tim{\times} \def\pl{\partial}

\begin{proof}
Let $\varphi\in\D{C}^2(\T^d)$ and set $u(x,t)=\eta^t*\varphi(x)$ for $(x,t)\in\R^d\tim(0+\infty)$ and $u(x,0)=\varphi(x)$ for $x\in\R^d$. It is a 
standard observation that  
for any multi-indices $\alpha$, with $|\alpha|\leq 2$, 
$D_x^\alpha u\in \D{C}(\R^d\tim[0+\infty))$ and $D_x^\alpha u(x,0)=D^\alpha\varphi(x)$
for all $x\in\R^d$, and that 
$\pl u/\pl t-\Delta u=0$ in $\R^d\tim[0+\infty)$. 
Observe that, for any $(x,q)\in \R^d\tim B_R$,
\[\begin{aligned}
\eta^\tau * \varphi(x+\tau q)-\varphi(x)&\,=\int_0^\tau\! \frac{d u}{dt}(x+t q,t)dt 
=\int_0^\tau\!\!  \left(\langle D_xu(x+tq,t),q\rangle+\frac{\pl u}{\pl t}(x+tq,t)\right) \! dt
\\&
\,=\int_0^\tau \big(\langle D_xu(x+tq,t),q\rangle+\Delta u(x+tq,t)\big)dt.
\end{aligned}\]
Now, setting 
\[
\omega(r)=\max_{(x,q,t)\in \R^d\tim B_R\tim [0,\,r]}\left|
\langle D_xu(x+tq,t)-D\varphi(x) ,q\rangle+\Delta u(x+tq,t)-\Delta \varphi(x)\right|,
\]
we have 
\[
\omega \in \D{C}([0+\infty)),\ \ \ \omega(0)=0, 
\]
and 
\[
\left|\frac{(\eta^\tau*\varphi)(x+\tau q) - \varphi(x)}{\tau}-\langle D\varphi(x),q\rangle -\Delta\varphi(x)\right|
\leq \omega(\tau)
\]
for all $(x,q)\in \T^d\times B_R$ and $\tau>0$.
\end{proof}

\begin{proof}[Proof of Theorem \ref{teo asymptotic}]
It follows from Proposition \ref{H-Holder} and Theorem \ref{teo eq viscous HJ} 
that there exists a unique pair $(\alpha,u)\in \R\times\cont(\T^d)$ with $u(x_0)=0$ such that $u$ is a viscosity solution to 
\begin{equation}\label{eq3 viscous HJ}
-\Delta u +H(x,Du)=\alpha\qquad\hbox{in $\T^d$}.
\end{equation}
In view of Proposition \ref{prop alpha_tau} and Proposition \ref{equi-con}, of the fact that $u_\tau(x_0)=0$ for all $\tau\in (0,1)$ and of Arzel\`a-Ascoli Theorem, 
we have that the set $\{(\alpha_\tau,u_\tau)\mid\tau\in (0,1)\}$ is precompact in $\R\times\cont(\T^d)$. 
In order to prove the assertion, it is therefore enough to show that, if the pair $(\alpha, u)$ is the limit of $(\alpha_{\tau_n},u_{\tau_n})$ in $\R\times\cont(\T^d)$ 
for some $\tau_n\to 0$, then $u$ is a solution to \eqref{eq3 viscous HJ}. 

Let us first show that such an $u$ is a viscosity subsolution to \eqref{eq3 viscous HJ}. Let $\varphi\in\cont^3(\T^d)$ be such that $u-\varphi$ has a strict maximum at $x_0$. Then there exists 
a sequence of points $(x_n)_n$ converging to $x_0$ in $\T^d$ such that $u_{\tau_n}-\varphi$ has a maximum at $x_n$. Let us set $\eps_n:=\max(u_{\tau_n}-\varphi)$ 
and $\varphi_n:=\varphi+\eps_n$. Then 
\[
 u_{\tau_n}\leq \varphi_n\quad\hbox{in $\T^d$}\qquad\hbox{and}\qquad u_{\tau_n}(x_n)=\varphi_n(x_n).
\]
By the monotone character of the operator $\Leg_{\tau_n}$ we infer 
\[
 \varphi_n(x_n)
 =
 u_{\tau_n}(x_n)
 =
 \Leg_{\tau_n}u_{\tau_n}(x_n)+\tau_n\alpha_{\tau_n}
 \leq
 \Leg_{\tau_n}\varphi_{n}(x_n)+\tau_n\alpha_{\tau_n},
 \]
hence, since $\varphi_n=\varphi+\eps_n$, 
\begin{equation}\label{eq1 asymptotic}
\frac{\varphi(x_n)-\Leg_{\tau_n}\varphi(x_n)}{\tau_n} \leq  \alpha_{\tau_n}
\end{equation}
By definition of $\Leg_{\tau_n}$, we infer that, for every fixed $q\in\R^d$, 
\[
\frac{\varphi(x_n)-\left(\eta^{\tau_n}*\varphi\right)(x_n+\tau_n q)}{\tau_n}-L(x_n,-q) \leq  \alpha_{\tau_n}
\]
By sending $n\to +\infty$ and by making use of Lemma \ref{lemma tool}, we end up with 
\[
 -\Delta \varphi(x_0)+\langle D\varphi(x_0),-q\rangle-L(x_0,-q) \leq \alpha.
\]
By taking the supremum of the above inequality with respect to $q\in\R^d$, we finally get, by the duality between $L$ and $H$, 
\[
 -\Delta\varphi(x_0)+H(x_0,D\varphi(x_0))\leq \alpha,
\]
thus showing that $u$ is a viscosity subsolution to \eqref{eq3 viscous HJ}. 

Let us now show that $u$ is a viscosity supersolution to \eqref{eq3 viscous HJ}. Let $\varphi\in\cont^3(\T^d)$ be such that $u-\varphi$ has a strict minimum at $x_0$. Then there exists 
a sequence of points $(x_n)_n$ converging to $x_0$ in $\T^d$ such that $u_{\tau_n}-\varphi$ has a minimum at $x_n$. Let us set $\eps_n:=\min(u_{\tau_n}-\varphi)$ 
and $\varphi_n:=\varphi+\eps_n$. Then 
\[
 u_{\tau_n}\geq \varphi_n\quad\hbox{in $\T^d$}\qquad\hbox{and}\qquad u_{\tau_n}(x_n)=\varphi_n(x_n).
\] 
By arguing analogously, we end up with 
\begin{equation*}
\frac{\varphi(x_n)-\Leg_{\tau_n}\varphi(x_n)}{\tau_n} \geq  \alpha_{\tau_n}
\end{equation*}
For each $n\in\N$, pick a minimizing $q_n\in\R^d$ for $\Leg_{\tau_n}\varphi(x_n)$, so that the previous inequality rereads as 
\begin{equation}\label{eq2 asymptotic}
\frac{\varphi(x_n)-\left(\eta^{\tau_n}*\varphi\right)(x_n+\tau_n q_n)}{\tau_n}-L(x_n,-q_n)\geq \alpha_{\tau_n}. 
\end{equation}
By making use of Proposition \ref{prop argmin L tau} and of the fact that 
$|D(\eta^\tau*\varphi)(x_n)|\leq \|D\varphi\|_\infty$, we infer that there exists $R>0$ such that $q_n\in B_R$ for every $n\in\N$. 
Up to extracting a further subsequence if necessary, we can assume that $q_n\to q$.  
Now we send $n\to +\infty$ in  \eqref{eq2 asymptotic} to get 
\[
 -\Delta \varphi(x_0)+\langle D\varphi(x_0),-q\rangle-L(x_0,-q) \geq \alpha.
\]
By the duality between $L$ and $H$, this implies   
\[
 -\Delta \varphi(x_0)+H(x_0, D\varphi(x_0)) \geq \alpha, 
\]
finally showing that $u$ is a viscosity supersolution to \eqref{eq3 viscous HJ}. 
\end{proof}
\subsection{Examples}\label{sez examples}  In this section, we exhibit some examples of Lagrangians for which the conclusion of Theorem \ref{teo asymptotic} holds true.\smallskip\\ 
\noindent{\bf Example 1:\quad}{\em $L\in \D{C}(\T^d\tim\R^d)$ satisfies {\bf (L1)}, {\bf (L2)} and $\min_{q\in\R^d} L(x,q)= c$ for all $x\in\T^d$ for some constant  $c\in\R$. 
}

This example includes the case when $L$ is independent of $x$, or the case $L(x,q)=a(x)|q|^m$ with $m\in (1,+\infty)$ and $a:\T^d\to (0,+\infty)$ Lipschitz continuous.  

In this case $\alpha_\tau=-c$ by Proposition \ref{prop alpha_tau} and $u_\tau\equiv 0$ for every $\tau>0$, so convergence of the $u_\tau$ trivially holds.\medskip   

\noindent{\bf Example 2:\quad}{\em $L(x,q)=L_0(q)+f(x)$\quad where $f\in \D{C}^{0,\gamma}(\T^d)$ with $\gamma\in (0,1)$ and $L_0\in \D{C}(\R^d)$ satisfies 
{\bf (L1)}, {\bf (L2)} and
\begin{equation}\label{condition semiconcavity}
D^2_q  L_0(q )\leq C_0 I_d \ \  \hbox{in}\  \R^d\setminus \overline B_{R_0} 
\end{equation}
in the sense of distributions, for some constants $C_0>0$ and $R_0>0$.}\smallskip

\indent This example includes the case $L_0(q ):=|q |^m$ with $m\in (1,2]$. Indeed, 
\[
DL_0(q )=m|q |^{m-2}q  \ \ \ \AND \ \ \ 
D^2L_0(q )=m|q |^{m-2}(I_d+(m-2)\bar q \otimes \bar q ) \leq m|q |^{m-2}I_d,
\]
where $\bar q =q /|q |$. The case $L_0(q ):=|q |^m+|q|$ with $m\in (1,2]$ is also included. 

\indent It is clear that $L$ satisfies {\bf (L1)} and  {\bf (L2)}. As for \eqref{L3}, first note that 
\[
f(x+h)+f(x-h)-2f(x)\leq |f(x+h)-f(x)|+|f(x-h)-f(x)|\leq 2\Lip_{\gamma}(f)|h|^\gamma.
\]
Condition \eqref{L3} is fulfilled in view of the following result:
%
%

\begin{lemma} 
Let $L_0\in \D{C}(\R^d)$ satisfy conditions {\bf (L1)}, {\bf (L2)} and \eqref{condition semiconcavity} for some constants $C_0>0$ and $R_0>0$. 
Then, for every $A>0$, there exists a constant  $C_A>0$ such that  
\begin{equation}\label{claim ex2}
L_0(q +k)+L_0(q -k)-2L_0(q )\leq C_A|k| \ \  \FORALL (q,k)\in \R^d\times \overline B_A.
\end{equation}
\end{lemma}
\begin{proof} 
Fix $A>0$. By the fact that the function $L_0$ is convex and locally bounded, 
we infer that it is Lipschitz on every ball in $\R^d$.  
In particular, there exists a constant $\tilde C_A>0$ such that 
\[
|L_0(q )-L_0(\eta)|\leq \tilde C_A  |q -\eta| \ \ \FORALL q,\eta\in \overline B_{R_0+2A}. 
\] 
From this, we get 
\[
L_0(q \pm k)-L_0(q )\leq \tilde C_A |k|\ \  \FORALL (q,k)\in \overline B_{R_0+A}\times \overline B_A. 
\]
Adding these two yields
\begin{equation}\label{eq example 2}
L_0(q +k)+L_0(q -k)-2L_0(q )\leq 2\tilde C_A |k|\ \  \FORALL (q,k)\in \overline B_{R_0+A}\times \overline B_{A}. 
\end{equation}
Let $(x_0,q _0)\in\T^d\tim (\R^d\setminus \overline B_{R_0+A})$. 
By \eqref{condition semiconcavity} we have 
\begin{equation}\label{ex2 semiconcavity}
D_q ^2 L_0(q )\leq C_0I_d \ \ \FORALL q \in \overline B_A(q _0)
\end{equation}
in the distributional sense. Let $\big(\rho_\eps\big)_{\eps>0}$ be a family of smooth mollifiers and set 
\[
L_\eps(q):=\int_{\R^d} \rho_\eps(\xi)L_0(\xi-q) d\xi\qquad\hbox{for all $q\in\R^d$.}
\] 
The function $L_{\eps}$ is smooth and, for $\eps>0$ small enough, satisfies \eqref{ex2 semiconcavity} pointwise  with the same constant $C_0>0$. 
By the Taylor theorem, 
for any $k\in \overline B_A$ we have 
\begin{align*}
L_\eps(q _0\pm k)-L_\eps(q _0)&\,=\du{D_q  L_\eps(q _0),\pm k}
+\frac 12 \int_0^1 (1-t) \du{D_q ^2L_\eps(q_0 \pm tk) k, k}dt
\\&\,\leq \du{D_q  L_\eps(q _0),\pm k}
+\frac 12 \int_0^1 C_0 |k|^2 dt.
\end{align*}
Adding these two yields
\[
L_\eps(q _0+k)+L_\eps(q _0-k)-2L_\eps(q _0)\leq C_0|k|^2
\leq C_0 A|k|.
\]
By sending $\eps\to 0$ we conclude that 
\[
L_0(q +k)+L_0(q-k)-2L_0(q )
\leq C_0 A|k|\quad\FORALL (q,k)\in (\R^d\setminus \overline B_{R_0+A})\times \overline B_A. 
\]
This combined with \eqref{eq example 2} implies claim \eqref{claim ex2} 
with $C_A:=\max\{C_0 A,\,2\tilde C_A\}$. 
\end{proof}

\bibliography{DiscreteHJ_8February2020}

\bibliographystyle{siam}

\end{document}